\documentclass[12pt,a4paper]{article}

\usepackage{amsmath}
\usepackage{amsfonts}
\usepackage{amssymb}
\usepackage{latexsym}

\usepackage{amsthm}

\usepackage{dsfont}

\makeatletter
\@addtoreset{equation}{section}
\makeatother

\newcommand{\be}{\begin{equation}}
\newcommand{\ee}{\end{equation}}
\newcommand{\nnr}{\nonumber \\}

\newcommand{\eq}[1]{(\ref{#1})}
\newcommand{\fr}{\frac}
\newcommand{\tf}{\tfrac}

\DeclareMathOperator{\df}{\! d \!}

\newcommand{\pd}{\partial}
\newcommand{\sr}{\sqrt}

\newcommand{\bb}[1]{\mathbb{#1}}

\renewcommand{\vec}[1]{\mathbf{#1}}
\newcommand{\gvec}[1]{\boldsymbol{#1}}

\newcommand{\mat}[1]{\mathsf{#1}}

\DeclareMathOperator{\Var}{Var}
\DeclareMathOperator{\Cov}{Cov}

\newtheorem{theorem}{Theorem}[section]
\newtheorem{proposition}[theorem]{Proposition}
\newtheorem{lemma}[theorem]{Lemma}
\newtheorem{corollary}[theorem]{Corollary}

\begin{document}


\thispagestyle{empty}

\begin{center}
\textbf{\Large{Properties of the Concrete distribution}}\\
\vspace{50pt}
\large{David D. K. Chow}
\end{center}

\begin{center}
{\bf Abstract\\}
\end{center}
We examine properties of the Concrete (or Gumbel-softmax) distribution on the simplex.  Using the natural vector space structure of the simplex, the Concrete distribution can be regarded as a transformation of the uniform distribution through a reflection and a location-scale transformation.  The Fisher information is computed and the corresponding information metric is hyperbolic space.  We explicitly give an explicit transformation of the parameters of the distribution to Poincar\'{e} half-space coordinates, which correspond to an orthogonal parameterization, and the Fisher--Rao geodesic distance is computed.

\newpage


\section{Introduction}


The Concrete distribution, also known as the Gumbel-softmax distribution, is a continuous probability distribution with support on the $(K - 1)$-dimensional probability simplex $S_K$ \cite{mamnte, jagupo}.  Like the more well-known Dirichlet distribution, it includes the uniform distribution on the simplex, but is otherwise distinct from the Dirichlet distribution.

There is a temperature parameter, named by analogy with the Boltzmann (or Gibbs) distribution in thermodynamics.  In the zero temperature limit, samples become concentrated at the corners of the simplex and the distribution approximates a categorical distribution: the limit of a continuous distribution approximates a discrete distribution, hence the portmanteau ``Concrete''.  There are additionally $K$ parameters that are easily interpreted as unnormalized probabilities for the $K$ categories in this limit; a single constraint reduces these to $K - 1$ independent parameters.

The Concrete distribution has applications in machine learning for stochastic neural networks with discrete random variables, such as random sampling from discrete distributions and estimation of parameter gradients through backpropagation.  One way of constructing the Concrete distribution is from combining Gumbel distributions through the softmax function, hence its alternative name of the Gumbel-softmax distribution.  The differentiability of the softmax function, unlike the argmax function, allows for backpropagation.  Samples from the Concrete distribution at low temperatures approximate samples from a categorical distribution.  In this context, the use of the Concrete distribution is one example of applying the Gumbel-max trick; see \cite{hukopavasl} for a recent review.

However, we shall instead emphasize a different perspective: the Concrete distribution is a transformation of the uniform distribution on the simplex.  More specifically, the Concrete distribution is, with respect to the natural vector space structure of the simplex \cite{bigufa98, bigufa, aitchison01}, a reflection composed with a location-scale transformation of the uniform distribution.  Some properties of the Concrete distribution become easier to understand when regarded as a transformation of a uniform distribution on a simplex, rather than through Gumbel distributions.  In some sense, including its extension to negative temperature, the Concrete distribution is the simplest family of distributions on the simplex, since it is obtained from the uniform distribution on the simplex by carrying out transformations related to the natural vector space structure of the simplex.

In this article, we explore basic properties of the Concrete distribution.  A more general family of probability distributions that includes both the Concrete distribution and the Dirichlet distribution has been called the (inverse) Schl\"{o}milch distribution \cite{chow}, and is essentially a distribution previously known \cite{craiu2, kobajo, graf} but with a change of sign of the temperature parameter.  Making such a generalization from the Concrete distribution is useful for computing quantities relating to the Concrete distribution itself.  Log-ratio means and log-ratio covariances are commonly used in compositional data analysis (see e.g.\ \cite{aitchison, paegto}), and for the inverse Schl\"{o}milch distribution can be explicitly computed \cite{chow}.  We shall go further here for the Concrete distribution, using the same methods to compute the Fisher information.

One way to present Fisher information is as a Riemannian metric, the (Fisher--Rao) information metric, which plays a key role in information geometry (see e.g.\ \cite{amari, ayjolesc, nielsen}).  An advantage of the geometric formulation is that changes of the parameterization of a distribution, which geometrically correspond to changes of coordinates, become more transparent.  For example, it becomes clearer to find orthogonal parameters, which have attractive statistical properties \cite{coxrei}, since these correspond geometrically to orthogonal coordinates.  The information metric can be used to define a geodesic distance, Fisher--Rao distance, between two distributions within the parameter family (see e.g.\ \cite{atkmit} for discussion and examples).  The geometric approach to Fisher information, in particular covariance under coordinate transformations, is also the motivation for the Jeffreys prior in Bayesian statistics.

Several basic families of distributions have particularly simple information metrics that are constant curvature.  For example, there are hyperbolic space information metrics for the univariate normal distribution and more general location-scale distributions \cite{sabobe}.  The geometry of hyperbolic space has been well understood for a long time, with known orthogonal coordinate systems.

We similarly find that the information metric of the Concrete distribution is hyperbolic space, and give an explicit transformation of the parameters of the distribution to demonstrate this.  Whereas hyperbolic space is maximally symmetric, in contrast the Dirichlet distribution has an information metric \cite{lebrprpu} that does not generically seem to have any continuous symmetries.  From the perspective of the information metric, the Concrete distribution is simpler than the Dirichlet distribution, because of the symmetries it inherits from being a reflection-location-scale transformation of a single distribution.

In Section 2, we review the Concrete distribution, emphasizing its derivation as a reflection-scale-location transformation, with respect to the natural vector space structure of the simplex, of the uniform distribution.  In Section 3, we consider aspects of the more general inverse Schl\"{o}milch distribution, obtaining log-ratio moments that will be used when computing the Fisher information of the Concrete distribution.  In Section 4, we collect these computations together, and examine the Fisher information from the geometric perspective of the information metric.


\section{Concrete distribution}


We first review the Concrete distribution \cite{mamnte, jagupo}, including its expression in terms of the softmax function and Gumbel distributions.  We then emphasize an alternative perspective in terms of a transformation of the uniform distribution, presenting a geometric explanation of some of its properties.

The $(K - 1)$-dimensional probability simplex $S_K$ can be regarded as embedded within $\bb R^K$ and is given by $S_K = \{ (x_1 , \ldots , x_K)  \subset \bb R^K : \sum_{i = 1}^K x_i = 1 , x_i \geq 0 \}$.  For practical computation, one may assume that $x_1 , \ldots , x_{K - 1}$ are independent, which then determines the fill-up value $x_K = 1 - \sum_{i = 1}^{K - 1} x_i$.  Integrals over $S_K$ may then be computed using the Lebesgue measure through
\be
\int_{S_K} \df {^{K - 1}} x = \int_0^1 \! \df x_1 \, \int_0^{1 - x_1} \! \df x_2 \ldots \int_0^{1 - x_1 - \ldots - x_{K - 2}} \! \df x_{K - 1} .
\ee
We use bold notation $\gvec \alpha = (\alpha_1 , \ldots , \alpha_K)$ to denote $K$-dimensional vectors.  The $K$-dimensional vector whose entries are all 1 is $\vec 1 = (1 , \ldots, 1)$.  The usual Kronecker delta is denoted $\delta_{i j}$, but we shall also use Kronecker delta symbols with additional indices, $\delta_{i j k} = \delta_{i j} \delta_{i k}$ and $\delta_{i j k l} = \delta_{i j} \delta_{i k} \delta_{i l}$, which take the value 1 if all indices are the same, otherwise taking the value 0.  $\vec e_m$ is a unit-vector in the $m$-direction, i.e.\ with $i$-component $(\vec e_m)_j = \delta_{m i}$, which we consider to be a column vector.


\subsection{Basic properties}


The Concrete distribution $C(\gvec \beta, \tau)$ has probability density function
\be
f(\vec x) = \fr{(K - 1)! \tau^{K - 1}}{(\sum_{j = 1}^K \beta_j/x_j^\tau)^K} \prod_{i = 1}^K \fr{\beta_i}{x_i^{\tau + 1}} ,
\label{concrete}
\ee
with support on the simplex $S_K$, $K \geq 2$.

The parameters $\beta_i$ are positive and the distribution is invariant under the scaling $\gvec \beta \rightarrow \lambda \gvec \beta$, $\lambda > 0$.  To uniquely identify a distribution within the family, a natural choice of constraint to impose is $\sum_{i = 1}^K \beta_K = 1$, so $\beta_i \in (0, 1)$ can then be considered as probabilities.  However, we shall not assume any particular constraint on the probability vector $\gvec \beta$ unless explicitly stated, since analytic calculations are easier when performed with manifest $S_K$ permutation symmetry, imposing a constraint on the probability vector $\gvec \beta$ only at the end.  

The parameter $\tau$ is positive, and is called the temperature.  In both the $\tau \rightarrow \infty$ and $\tau \rightarrow 0$ limits, the Concrete distribution, which is continuous, tends to a discrete distribution.  In the $\tau \rightarrow \infty$ limit, the Concrete distribution becomes highly peaked around $K^{-1} \vec 1$, which is the centre of the simplex, tending to a Dirac distribution $f(\vec x) = \delta(\vec x - K^{-1} \vec 1)$, i.e.\ the Concrete distribution degenerates to a deterministic one-point distribution.  In the $\tau \rightarrow 0$ limit, the Concrete distribution tends towards a categorical distribution: samples tend towards one-hot vectors $\vec e_i$.  It has been shown that, in this limit, the categorical distribution has probabilities
\be
\bb P \Big( \lim_{\tau \rightarrow 0} X_i = 1 \Big) = \fr{\beta_i}{\sum_{j = 1}^K \beta_j} ,
\label{pcat}
\ee
which we shortly rederive.

It is convenient to decompose the Concrete probability density function \eq{concrete} as
\be
f(\vec x) = \fr{h(\vec x)}{J_0} ,
\ee
where the functions
\begin{align}
h(\vec x) & = \fr{1}{k(\vec x)^K \prod_{i = 1}^K x_i^{\tau + 1}} , & k(\vec x) & = \sum_{j = 1}^K \fr{\beta_k}{x_j^\tau}
\label{hk}
\end{align}
depend on the coordinates $\vec x$, and $J_0$ is a normalization constant that can be expressed as
\be
J_0 = \fr{1}{(K - 1)! \tau^{K - 1} \prod_{i = 1}^K \beta_i}
\label{J01}
\ee
or
\be
J_0 = \int_{S_K} \! \df {^{K - 1}} x \, h(\vec x) .
\label{J02}
\ee

A Concrete distribution $\vec X$ can be expressed as
\be
X_i = \textrm{softmax} [(\vec W + \log \gvec \beta)/\tau]_i , 
\ee
where $W_i \sim \textrm{Gumbel}(0, 1)$ is a standard Gumbel distribution, $\log \gvec \beta = (\log \beta_1 , \ldots , \log \beta_K)$ and the softmax function is
\be
\textrm{softmax}(\vec x)_i = \fr{\exp (x_i)}{\sum_{j = 1}^K \exp (x_j)} .
\ee
This formulation gives the Concrete distribution its alternative name of the Gumbel-softmax distribution.  It is simple to sample from a standard Gumbel distribution $W$, since $W = - \log Z$, where $Z \sim \textrm{Exponential}(1)$.  This unit exponential distribution can be simply sampled from a uniform distribution through an inverse transform, although there are also more sophisticated techniques.  In the $\tau \rightarrow 0$ limit, $\textrm{softmax}(\vec x/\tau)$ tends to $\textrm{argmax}_i (\vec x)$, represented as a one-hot vector; the differentiability of softmax, as opposed to argmax, enables backpropagation in stochastic networks.


\subsection{Transforming the uniform distribution}


As is widely known in the compositional data analysis literature, the simplex admits a vector space structure \cite{bigufa98, bigufa, aitchison01}, sometimes known as Aitchison geometry.  One defines a closure operator $\mathcal{C}$, which normalizes a vector $\vec x$ in the positive orthant $\bb R_+^K$ so that it lies on the simplex $S_K$, i.e.\ $\mathcal{C} \vec x = \vec x/\sum_{j = 1}^K x_j$.  The analogue in $S_K$ of vector addition in $\bb R^K$ is the perturbation operator \cite{aitchison82}
\be
\vec x \oplus \vec y = \mathcal{C} (x_1 y_1 , \ldots , x_K y_K) ,
\label{plus}
\ee
i.e.\ the Hadamard product of the vectors $\vec x$ and $\vec y$, followed by the closure operator.  The perturbation of $\vec x$ and $\vec y$ is often described as a perturbation of $\vec x$ by $\vec y$.  The analogue in $S_K$ of scalar multiplication in $\bb R^K$ is the powering operator \cite{aitchison}
\be
\alpha \odot \vec x = \mathcal{C} (x_1^\alpha, \ldots , x_K^\alpha) ,
\label{times}
\ee
where $\alpha \in \bb R$.  The operators $\oplus$ and $\odot$ together provide a vector space structure on $S_K$.

The definitions of the operators $\oplus$ and $\odot$ may appear unmotivated.  However, Leinster has explained that these operators are naturally induced by the usual vector space structure of $\bb R^K$ (see pp.\ 124--126 of \cite{leinster}).  First, one maps $\bb R^K$ to the positive orthant $\bb R_+^K$ through an exponential map applied to each coordinate.  Then, one considers the space of rays, using an equivalence relation $\sim$ that identifies $\vec x \sim \lambda \vec x$ for any $\lambda > 0$.  Under these mappings from $\bb R^K$ to $\bb R_+^K / \sim$, the operations of vector addition and scalar multiplication on $\bb R^K$ become the operators $\oplus$ and $\odot$ but with the closure operator $\mathcal{C}$ omitted.  Finally, one identifies rays with points on the simplex $S_K$, giving the operators $\oplus$ and $\odot$ on $S_K$ as defined by \eq{plus} and \eq{times}.

A Concrete distribution $\vec X$ is related to the uniform distribution $\vec Y$ on the simplex by \cite{chow}
\begin{align}
Y_i & = \fr{\beta_i}{X_i^\tau \sum_{j = 1}^K \beta_j/X_j^\tau} , & X_i & = \fr{\beta_i^{1/\tau}}{Y_i^{1/\tau} \sum_{j = 1}^K (\beta_j/Y_j)^{1/\tau}} .
\label{transform}
\end{align}
This transformation on the simplex $S_K$ is analogous to random variables $\vec X$ and $\vec Y$ on $\bb R^K$ transforming as
\begin{align}
Y_i & = \log \beta_i - \tau X_i , & X_i & = (\log \beta_i - Y_i)/\tau ,
\end{align}
which is a composition of a reflection and a location-scale transformation.

We may also considering the opposite sign of the temperature parameter $\tau$.  The simplex transformation \eq{transform} can be generalized to
\begin{align}
Y_i & = \fr{\beta_i X_i^{\pm \tau}}{\sum_{j = 1}^K \beta_j X_j^{\pm \tau}} , & X_i & = \fr{\beta_i^{1/\tau} Y_i^{\pm 1/\tau}}{\sum_{j = 1}^K \beta_j Y_j^{\pm 1/\tau}} ,
\end{align}
where $\tau > 0$, which is analogous to the $\bb R^K$ transformation
\begin{align}
Y_i & = \log \beta_i \pm \tau X_i , & X_i & = \pm (Y_i - \log \beta_i)/\tau .
\end{align}
Upper signs correspond to a scale-location transformation on $\bb R^K$, without any reflection.  The two signs may be unified by defining a real-valued inverse temperature $B = \pm 1/\tau$, as is standard in thermodynamics.  The distributions generated by the powering transformation are known in thermodynamics as escort distributions (see Chapter 9 of \cite{becsch}).


\subsection{Rounding}


The simplex transformation \eq{transform}, from a Concrete distribution $\vec X$ to a uniform distribution $\vec Y$, implies that the rounding probability
\be
p_i = \bb P [X_i = \max (X_1 , \ldots , X_K)]
\label{probmaxX}
\ee
can be expressed as
\be
p_i = \bb P [\beta_i/Y_i = \max (\beta_1/Y_1 , \ldots , \beta_K/Y_K)] = \bb P [Y_i/\beta_i = \min (Y_1/\beta_1 , \ldots , Y_K/\beta_K)] .
\label{probmaxY}
\ee
The uniform distribution $\vec Y$ is independent of the temperature $\tau$, so the probability $p_i$ is independent of the temperature $\tau$.  Taking the limit $\tau \rightarrow 0$, we obtain the categorical distribution with probabilities $p_i$, as stated in \eq{pcat}.

There are several methods to derive the known value
\be
p_i = \fr{\beta_i}{\sum_{j = 1}^K \beta_j} .
\label{pvalue}
\ee
By considering the expression \eq{probmaxY}, which is in terms of a uniform distribution, the probability can be regarded as a calculation of volumes in Euclidean geometry.  The simplex $S_K$, is the convex hull of the vertices $V = \{ \vec e_1 , \ldots , \vec e_K \}$.  We introduce an extra vertex at $\mathcal{C}(\gvec \beta) = \gvec \beta /\sum_{j = 1}^K \beta_j$ to partition the simplex $S_K$ into subsimplices $S_{K, i}$, where $i = 1 , \ldots , K$.  The subsimplex $S_{K, i}$ is the convex hull of the vertices $V_i = V \cup \{ \mathcal{C} (\gvec \beta) \} \setminus \{ \vec e_i \}$, i.e.\ the vertices $V$, but with $\vec e_i$ replaced by $\mathcal{C} (\gvec \beta)$.  The boundary between subsimplices $S_{K, i}$ and $S_{K, j}$ is a $(K - 2)$-dimensional simplex given by the convex hull of the vertices $V \cup \{ \mathcal{C}(\gvec \beta) \} \setminus \{ \vec e_i , \vec e_j \}$, and is part of the $(K - 2)$-dimensional plane $y_i/\beta_i = y_j/\beta_j$.  The probability $p_i$ is therefore the ratio of the $(K - 1)$-dimensional volumes of $S_{K, i}$ and $S_K$.  The $K \times K$ matrix
\be
\mat M_i = \begin{pmatrix}
\vec e_1 & \cdots & \vec e_{i - 1} & \mathcal{C} (\gvec \beta) & \vec e_{i + 1} & \cdots & \vec e_K
\end{pmatrix}
\ee
represents an affine transformation that maps the vertices $V$ to the vertices $V_i$.  The determinant $\det \mat M_i = \beta_i/\sum_{j = 1}^K \beta_j$ represents the ratio of the volumes of $S_{K, i}$ and $S_K$, so the value of the probability \eq{pvalue} follows.  Similar geometric reasoning has been used in \cite{vaugab, dempster72, jagoedde}.  There are also derivations that rely on considering a race between competing exponential clocks and using standard properties of exponential distributions \cite{geheyezh, bunasc13, bunasc18, maddison}.


\section{Inverse Schl\"{o}milch distribution}


We now consider a generalization of the Concrete distribution that enables us to compute quantities relating to the Concrete distribution itself.  The inverse Schl\"{o}milch distribution \cite{chow} adds a vector of positive parameters $\gvec \alpha$ to incorporate (if the temperature is negative) the Dirichlet distribution.

The inverse Schl\"{o}milch distribution $IS(\gvec \alpha, \gvec \beta, \tau)$ has probability density function
\be
f(\vec x) = \fr{\tau^{K - 1} \Gamma (\alpha_+)}{(\sum_{j = 1}^K \beta_j/x_j^\tau)^{\alpha_+}} \prod_{i = 1}^K \fr{\beta_i^{\alpha_i}}{\Gamma(\alpha_i) x_i^{\tau \alpha_i + 1}} ,
\label{IS}
\ee
where $\alpha_i > 0$ and $\alpha_+ = \sum_{i = 1}^K \alpha_i$.  The components of the probability vector $\gvec \beta$ are again positive, and there remains an invariance under rescaling $\gvec \beta \rightarrow \lambda \gvec \beta$, $\lambda > 0$.

The case $\tau < 0$ (equivalent to replacing $\tau \rightarrow - \tau$ and keeping $\tau$ positive) is of separate interest, and has been previously constructed from generalized gamma distributions and considered in \cite{craiu2, kobajo, graf}, and called the Schl\"{o}milch distribution in \cite{chow}.  However, we continue to consider only $\tau > 0$, because of its inclusion of the Concrete distribution \eq{concrete}, which is obtained by setting the Dirichlet vector $\gvec \alpha = \vec 1$.

Although we are primarily concerned with the Concrete distribution here, the more general inverse Schl\"{o}milch is useful for computing properties of the Concrete distribution.  In particular, because the Dirichlet distribution is an exponential family distribution, the inverse Schl\"{o}milch distribution is an exponential family distribution if the probability vector $\gvec \beta$ and temperature $\tau$ are regarded as fixed.

It is convenient to decompose the inverse Schl\"{o}milch probability density function \eq{IS} as
\be
f(\vec x) = \fr{g(\vec x)}{J (\gvec \alpha)} ,
\ee
where the function
\be
g(\vec x) = \fr{1}{(\sum_{j = 1}^K \beta_j/x_j^\tau)^{\alpha_+}  \prod_{i = 1}^K x_i^{\tau \alpha_i + 1}}
\label{gx}
\ee
depends on the coordinates $\vec x$, and $J(\gvec \alpha)$ is a normalization constant that can be expressed as
\be
J(\gvec \alpha) = \fr{1}{\tau^{K - 1} \Gamma (\alpha_+)} \prod_{i = 1}^K \fr{\Gamma(\alpha_i)}{\beta_i^{\alpha_i}} 
\label{Jalpha1}
\ee
or
\be
J(\gvec \alpha) = \int_{S_K} \! \df {^{K - 1}} x \, g(\vec x) = \int_{S_K} \! \fr{\df {^{K - 1}} x}{(\sum_{j = 1}^K \beta_j/x_j^\tau)^{\alpha_+} \prod_{i = 1}^K x_i^{\tau \alpha_i + 1}} .
\label{Jalpha2}
\ee
Evaluated at $\gvec \alpha = \vec 1$, the function $g(\vec x)$ of \eq{gx} reduces to $h(\vec x)$ of \eq{hk}, and $J(\vec 1) = J_0$.


\subsection{Exponential family}


If all parameters except the Dirichlet vector $\gvec \alpha$ are fixed, then the inverse Schl\"{o}milch distribution is a member of the exponential family of distributions, i.e.\ the probability density functions can be expressed in the form
\be
f(\vec x) = \fr{t(\vec x)}{J(\gvec \alpha)} \exp \bigg( \sum_{i = 1}^K \alpha_i T_i (\vec x) \bigg) .
\label{expfamily}
\ee
For the inverse Schl\"{o}milch distribution, with $f(\vec x)$ given by \eq{IS}, we may take the sufficient statistic as
\be
\vec T (\vec X) = - \tau (\log X_1 , \ldots , \log X_K) - \log [k(\vec X)] \vec 1 ,
\label{sufficient}
\ee
the normalization constant $J(\gvec \alpha)$ as given before in \eq{Jalpha1}, and
\be
t(\vec x) = \fr{1}{\prod_{i = 1}^K x_i} .
\ee

In the general form of an exponential family distribution \eq{expfamily}, the natural parameter is $\gvec \alpha$, and $\vec T(\vec X)$ is a sufficient statistic.  It is well-known that the means are
\be
\bb E (T_i) = \fr{\pd \log J(\gvec \alpha)}{\pd \alpha_i} ,
\ee
the covariances are
\be
\Cov (T_i, T_j) = \fr{\pd^2 \log J(\gvec \alpha)}{\pd \alpha_i \, \pd \alpha_j} ,
\ee
and the raw second moments are
\be
\bb E (T_i T_j) = \fr{\pd^2 \log J(\gvec \alpha)}{\pd \alpha_i \, \pd \alpha_j} + \fr{\pd \log J(\gvec \alpha)}{\pd \alpha_i} \fr{\pd \log J(\gvec \alpha)}{\pd \alpha_j} .
\ee
Using the sufficient statistic \eq{sufficient} for the inverse Schl\"{o}milch distribution, we obtain the log-ratio means
\be
\bb E \bigg[ \log \bigg( \fr{X_i}{X_k} \bigg) \bigg] = - \fr{1}{\tau} \bigg( \fr{\pd}{\pd \alpha_i} - \fr{\pd}{\pd \alpha_k} \bigg) \log J
\label{islrmean}
\ee
and the log-ratio covariances
\be
\Cov \bigg[ \log \bigg( \fr{X_i}{X_k} \bigg) , \log \bigg( \fr{X_j}{X_l} \bigg) \bigg] = \fr{1}{\tau^2} \bigg( \fr{\pd}{\pd \alpha_i} - \fr{\pd}{\pd \alpha_k} \bigg) \bigg( \fr{\pd}{\pd \alpha_j} - \fr{\pd}{\pd \alpha_l} \bigg) \log J ,
\label{islrcov}
\ee
through which we obtain the raw log-ratio second moments
\begin{align}
\bb E \bigg[ \log \bigg( \fr{X_i}{X_k} \bigg) \log \bigg( \fr{X_j}{X_l} \bigg) \bigg] & = \Cov \bigg[ \log \bigg( \fr{X_i}{X_k} \bigg) , \log \bigg( \fr{X_j}{X_l} \bigg) \bigg] + \bb E \bigg[ \log \bigg( \fr{X_i}{X_k} \bigg) \bigg] \bb E \bigg[ \log \bigg( \fr{X_j}{X_l} \bigg) \bigg] \nnr
& = \fr{1}{\tau^2} J^{-1} \bigg( \fr{\pd}{\pd \alpha_i} - \fr{\pd}{\pd \alpha_k} \bigg) \bigg( \fr{\pd}{\pd \alpha_j} - \fr{\pd}{\pd \alpha_l} \bigg) J ,
\label{lrsecond}
\end{align}
where the final equality is verified by expansion of \eq{islrmean} and \eq{islrcov}.


\subsection{Log-ratio moments}


For probability distributions on the simplex, it is common, particularly in compositional data analysis, to consider expectations involving log-ratios $\log(X_i/X_k)$ \cite{aitchison}.  For the inverse Schl\"{o}milch distribution, we have seen that log-ratio moments naturally arise from considering the distribution as a member of the exponential family of distributions.  We record these log-ratio moments here.


\subsubsection{Log-ratio means}


\begin{proposition}
\label{ISmeansprop}
The log-ratio means of the inverse Schl\"{o}milch distribution $IS(\gvec \alpha, \gvec \beta, \tau)$ are
\be
\bb E \bigg[ \log \bigg( \fr{X_i}{X_k} \bigg) \bigg] = \fr{1}{\tau} \bigg[ - \psi(\alpha_i) + \psi(\alpha_k) + \log \bigg( \fr{\beta_i}{\beta_k} \bigg) \bigg] ,
\label{Elogik}
\ee
where $\psi(x) = \Gamma'(x)/\Gamma(x)$ is the digamma function.
\end{proposition}

\begin{proof}
From $J(\gvec \alpha)$, given by \eq{Jalpha1}, we have
\be
\bigg( \fr{\pd}{\pd \alpha_i} - \fr{\pd}{\pd \alpha_k} 
\bigg) \log J = \psi(\alpha_i) - \psi(\alpha_k) - \log \beta_i + \log \beta_k ,
\label{JdikJ1}
\ee
which we substitute into the general expression for log-ratio means \eq{islrmean}.
\end{proof}

\begin{corollary}
The log-ratio means of the Concrete distribution $C(\gvec \beta, \tau)$ are
\be
\bb E \bigg[ \log \bigg( \fr{X_i}{X_k} \bigg) \bigg] = \fr{1}{\tau} \log \bigg( \fr{\beta_i}{\beta_k} \bigg) .
\ee
\end{corollary}

\begin{proof}
Set $\gvec \alpha = \vec 1$ in \eq{Elogik} and use the value $\psi(1) = - \gamma$, where $\gamma$ is the Euler--Mascheroni constant.
\end{proof}


\subsubsection{Log-ratio covariances}


\begin{proposition}
The log-ratio covariances of the inverse Schl\"{o}milch distribution $IS(\gvec \alpha, \gvec \beta, \tau)$ are
\be
\Cov \bigg[ \log \bigg( \fr{X_i}{X_k} \bigg) , \log \bigg( \fr{X_j}{X_l} \bigg) \bigg] = \fr{1}{\tau^2} [(\delta_{i j} - \delta_{i l}) \psi'(\alpha_i) - (\delta_{k j} - \delta_{k l}) \psi'(\alpha_k)] ,
\label{Covikjl}
\ee
where $\psi'(x)$ is the trigamma function.
\end{proposition}

\begin{proof}
By differentiating \eq{JdikJ1}, we have
\begin{align}
\bigg( \fr{\pd}{\pd \alpha_i} - \fr{\pd}{\pd \alpha_k} \bigg) \bigg( \fr{\pd}{\pd \alpha_j} - \fr{\pd}{\pd \alpha_l} \bigg) \log J & = (\delta_{i j} - \delta_{i l}) \psi'(\alpha_i) - (\delta_{k j} - \delta_{k l}) \psi'(\alpha_k) ,
\end{align}
which we substitute into the general expression for log-ratio covariances \eq{islrcov}.
\end{proof}

\begin{corollary}
The log-ratio covariances of the Concrete distribution $C(\gvec \beta, \tau)$ are
\be
\Cov \bigg[ \log \bigg( \fr{X_i}{X_k} \bigg) , \log \bigg( \fr{X_j}{X_l} \bigg) \bigg] = \fr{\pi^2}{6 \tau^2} (\delta_{i j} - \delta_{i l} - \delta_{j k} + \delta_{k l}) .
\label{CovConcrete}
\ee
\end{corollary}

\begin{proof}
Set $\gvec \alpha = \vec 1$ in \eq{Covikjl} and use the value $\psi'(1) = \pi^2/6$.
\end{proof}

In general, log-ratio covariances can be expressed in terms of log-ratio variances $\tau_{i j} = \Var [\log (X_i/X_j)]$ through \cite{aitchison84}
\be
\Cov \bigg[ \log \bigg( \fr{X_i}{X_k} \bigg) , \log \bigg( \fr{X_j}{X_l} \bigg) \bigg] = \fr{1}{2} (\tau_{i l} + \tau_{j k} - \tau_{i j} - \tau_{k l}) .
\ee

By setting $i = j$ and $k = l$ in \eq{Covikjl}, we obtain the following result.
\begin{corollary}
The log-ratio variances of the inverse Schl\"{o}milch distribution $IS(\gvec \alpha, \gvec \beta, \tau)$ are
\be
\Var \bigg[ \log \bigg( \fr{X_i}{X_k} \bigg) \bigg] = \fr{1}{\tau^2} (1 - \delta_{i k}) [ \psi'(\alpha_i) + \psi'(\alpha_k)] .
\ee
\end{corollary}

From \eq{CovConcrete}, we similarly obtain the following result.
\begin{corollary}
The log-ratio variances of the Concrete distribution $C(\gvec \beta, \tau)$ are
\be
\Var \bigg[ \log \bigg( \fr{X_i}{X_k} \bigg) \bigg] = \fr{\pi^2}{3 \tau^2} (1 - \delta_{i k}) .
\ee
\end{corollary}


\subsection{$\gvec \alpha = \vec 1 + \vec e_m + \vec e_n$ case}


For later use in computing the Fisher information of the Concrete distribution, we find log-ratio moments for the special case that the Dirichlet vector is $\gvec \alpha = \vec 1 + \vec e_m + \vec e_n$.

We first note that
\begin{align}
J(\vec 1 + \vec e_m + \vec e_n) & = \fr{J_0 (1 + \delta_{m n})}{K (K + 1) \beta_m \beta_n} ,
\label{Jspecial}
\end{align}
where $J_0$ is given by \eq{J01}.  From th values of the digamma function $\psi(2) = 1 - \gamma$ and $\psi(3) = 3/2 - \gamma$, we have
\begin{align}
\psi((\vec 1 + \vec e_m + \vec e_n)_i) & = \delta_{i m} + \delta_{i n} - \tf{1}{2} \delta_{i m n} - \gamma .
\label{digamma}
\end{align}
From the values of the trigamma function $\psi'(2) = \pi^2/6 - 1$ and $\psi'(3) = \pi^2/6 - 5/4$, we have
\begin{align}
\psi'((\vec 1 + \vec e_m + \vec e_n)_i) & = \pi^2/6 - \delta_{i m} - \delta_{i n} + \tf{3}{4} \delta_{i m n} .
\label{trigamma}
\end{align}

The log-ratio means are a special case of those found previously.
\begin{corollary}
The log-ratio means of the inverse Schl\"{o}milch distribution $IS(\vec 1 + \vec e_m + \vec e_n, \gvec \beta, \tau)$ are
\begin{align}
\bb E \bigg[ \log \bigg( \fr{X_i}{X_k} \bigg) \bigg] & = \fr{K (K + 1) \beta_m \beta_n}{J_0(1 + \delta_{m n})} \int_{S_K} \df{^{K - 1}} x \, \fr{h(\vec x) (\log x_i - \log x_k)}{k(\vec x)^2 x_m^\tau x_n^\tau} \nnr
& = \fr{1}{\tau} \bigg[ - \delta_{i m} - \delta_{i n} + \delta_{k m} + \delta_{k n} - \tf{1}{2} \delta_{i m n} + \tf{1}{2} \delta_{k m n} + \log \bigg( \fr{\beta_i}{\beta_k} \bigg) \bigg] .
\label{dikJ}
\end{align}
\end{corollary}

\begin{proof}
The first equality comes from $J(\vec 1 + \vec e_m + \vec e_n)$ in \eq{Jspecial} and from definitions.  For the second equality, set $\gvec \alpha = \vec 1 + \vec e_m + \vec e_n$ in \eq{Elogik} of Proposition \ref{ISmeansprop}, and use the value of the digamma function in \eq{digamma}, i.e.\ $\psi(\alpha_i) = \delta_{i m} + \delta_{i n} - \tf{1}{2} \delta_{i m n} - \gamma $.
\end{proof}

For later use, we will need raw log-ratio second moments.
\begin{lemma}
The raw log-ratio second moments $\bb E [ \log ( X_i/X_k ) \log ( X_i/X_l) ]$ of the inverse Schl\"{o}milch distribution $IS(\vec 1 + \vec e_m + \vec e_n, \gvec \beta, \tau)$ are
\begin{align}
& \bb E \bigg[ \log \bigg( \fr{X_i}{X_k} \bigg) \log \bigg( \fr{X_i}{X_l} \bigg) \bigg] = \fr{K (K + 1) \beta_m \beta_n}{J_0(1 + \delta_{m n})} \int_{S_K} \df{^{K - 1}} x \, \fr{h(\vec x) (\log x_i - \log x_k) (\log x_i - \log x_l)}{k(\vec x)^2 x_m^\tau x_n^\tau} \nnr
& = \tau^{-2} [\delta_{i m n} + \delta_{i l m n} + \delta_{i k m n} - \delta_{k l m n} - \delta_{i m} \delta_{k n} - \delta_{i m} \delta_{l n} - \delta_{i n} \delta_{k m} - \delta_{i n} \delta_{l m} + \delta_{k m} \delta_{l n} + \delta_{k n} \delta_{l m} \nnr
& \qquad - (2 \delta_{i m} + 2 \delta_{i n} - \delta_{k m} - \delta_{k n} - \delta_{l m} - \delta_{l n} - \delta_{i m n} + \tf{1}{2} \delta_{k m n} + \tf{1}{2} \delta_{l m n}) \log \beta_i \nnr
& \qquad + (\delta_{i m} + \delta_{i n} - \delta_{k m} - \delta_{k n} - \tf{1}{2} \delta_{i m n} + \tf{1}{2} \delta_{k m n}) \log \beta_l \nnr
& \qquad + (\delta_{i m} + \delta_{i n} - \delta_{l m} - \delta_{l n} - \tf{1}{2} \delta_{i m n} + \tf{1}{2} \delta_{l m n}) \log \beta_k \nnr
& \qquad + (\log \beta_i - \log \beta_k) (\log \beta_i - \log \beta_l) + (1 - \delta_{i k} - \delta_{i l} + \delta_{k l}) \pi^2/6] .
\label{dikilJ}
\end{align}
\end{lemma}

\begin{proof}
Again, $J(\vec 1 + \vec e_m + \vec e_n)$ in \eq{Jspecial} and definitions give the first equality.  We compute second derivatives of $J(\gvec \alpha)$, given by \eq{Jalpha1}, with respect to $\gvec \alpha$ through
\be
J^{-1} \fr{\pd^2 J}{\pd \alpha_i \, \pd \alpha_j} = \fr{\pd^2 
\log J}{\pd \alpha_i \, \pd \alpha_j} + \fr{\pd \log J}{\pd \alpha_i} \fr{\pd \log J}{\pd \alpha_j} .
\ee
Therefore,
\begin{align}
J^{-1} \fr{\pd^2 J}{\pd \alpha_i \, \pd \alpha_j} & = \tf{1}{2} \delta_{i j} [\psi'(\alpha_i) + \psi'(\alpha_j)] + [\psi (\alpha_i) - \log \beta_i] [\psi (\alpha_j) - \log \beta_j] \nnr
& \qquad - \psi'(\alpha_+) + \psi(\alpha_+) [\psi(\alpha_+) - \psi(\alpha_i) - \psi(\alpha_j) + \log \beta_i + \log \beta_j] ,
\end{align}
and so
\begin{align}
& J^{-1} \bigg( \fr{\pd}{\pd \alpha_i} - \fr{\pd}{\pd \alpha_k} \bigg) \bigg( \fr{\pd}{\pd \alpha_j} - \fr{\pd}{\pd \alpha_l} \bigg) J \nnr
& = \tf{1}{2} \psi'(\alpha_i) (\delta_{i j} - \delta_{i l}) + \tf{1}{2} \psi'(\alpha_j) (\delta_{i j} - \delta_{j k}) - \tf{1}{2} \psi'(\alpha_k) (\delta_{j k} - \delta_{k l}) - \tf{1}{2} \psi'(\alpha_l) (\delta_{i l} - \delta_{k l}) \nnr
& \qquad + [\psi (\alpha_i) - \psi(\alpha_k) - \log \beta_i + \log \beta_k] [\psi (\alpha_j) - \psi(\alpha_l) - \log \beta_j + \log \beta_l] .
\end{align}
We then evaluate at $\gvec \alpha = \vec 1 + \vec e_m + \vec e_n$, using the values of the digamma and trigamma functions in \eq{digamma} and \eq{trigamma}, i.e.\ $\psi(\alpha_i) = \delta_{i m} + \delta_{i n} - \tf{1}{2} \delta_{i m n} - \gamma $ and $\psi'(\alpha_i) = \pi^2/6 - \delta_{i m} - \delta_{i n} + \tf{3}{4} \delta_{i m n}$, so
\begin{align}
& J^{-1} \bigg( \fr{\pd}{\pd \alpha_i} - \fr{\pd}{\pd \alpha_k} \bigg) \bigg( \fr{\pd}{\pd \alpha_j} - \fr{\pd}{\pd \alpha_l} \bigg) J \nnr
& = - \delta_{i j m n} + \delta_{i l m n} + \delta_{j k m n} - \delta_{k l m n} + (\delta_{i m} - \delta_{k m}) (\delta_{j n} - \delta_{l n}) + (\delta_{i n} - \delta_{k n}) (\delta_{j m} - \delta_{l m})  \nnr
& \qquad + (\delta_{i m} + \delta_{i n} - \delta_{k m} - \delta_{k n} - \tf{1}{2} \delta_{i m n} + \tf{1}{2} \delta_{k m n}) (- \log \beta_j + \log \beta_l) \nnr
& \qquad + (\delta_{j m} + \delta_{j n} - \delta_{l m} - \delta_{l n} - \tf{1}{2} \delta_{j m n} + \tf{1}{2} \delta_{l m n}) (- \log \beta_i + \log \beta_k) \nnr
& \qquad + (\log \beta_i - \log \beta_k) (\log \beta_j - \log \beta_l) + (\delta_{i j} - \delta_{k j} - \delta_{i l} + \delta_{k l}) \pi^2/6 .
\end{align}
We now use the general expression for raw log-ratio second moments \eq{lrsecond} and set $i = j$.
\end{proof}

\begin{corollary}
The raw log-ratio second moments $\bb E [ \log ( X_i/X_k )^2 ]$ of the inverse Schl\"{o}milch distribution $IS(\vec 1 + \vec e_m + \vec e_n, \gvec \beta, \tau)$ are
\begin{align}
& \bb E \bigg[ \log \bigg( \fr{X_i}{X_k} \bigg) ^2 \bigg] = \fr{K (K + 1) \beta_m \beta_n}{J_0(1 + \delta_{m n})} \int_{S_K} \df{^{K - 1}} x \, \fr{h(\vec x) (\log x_i - \log x_k)^2}{k(\vec x)^2 x_m^\tau x_n^\tau} \nnr
& = \tau^{-2} [\delta_{i m n} + \delta_{k m n} + 2 \delta_{i k m n} - 2 \delta_{i m} \delta_{k n} - 2 \delta_{i n} \delta_{k m} - (2 \delta_{i m} + 2 \delta_{i n} - 2 \delta_{k m} - 2 \delta_{k n} - \delta_{i m n} + \delta_{k m n}) \nnr
& \qquad \times (\log \beta_i - \log \beta_k) + (\log \beta_i - \log \beta_k)^2 + (1 - \delta_{i k}) \pi^2/3 ] .
\label{dikikJ}
\end{align}
\end{corollary}

\begin{proof}
Set $k = l$ in \eq{dikilJ}.
\end{proof}


\section{Fisher information and the information metric}


Consider a family of probability distributions, with probability density function $f(\vec x)$, parameterized by $K$ independent continuous parameters $\eta_1 , \ldots , \eta_K$.  The Fisher information can be expressed in terms of matrix components
\be
\mathcal{I}_{a b} = \bb E \bigg( \fr{\pd \log f}{\pd \eta_a} \fr{\pd \log f}{\pd \eta_b} \bigg) = - \bb E \bigg( \fr{\pd^2 \log f}{\pd \eta_a \, \pd \eta_b} \bigg) .
\ee
where $a, b = 1, \ldots , K$, and the second equality relies on sufficient regularity.  The Fisher information can equivalently be expressed geometrically through a Riemannian metric as the (Fisher--Rao) information metric
\be
\df s^2 = \sum_{a = 1}^K \sum_{b = 1}^K \mathcal{I}_{a b} \, \df \eta_a \, \df \eta_b .
\ee

Suppose that the $K$ independent parameters are obtained by applying one constraint to $K + 1$ parameters $\theta_1 , \ldots , \theta_{K + 1}$.  An advantage of the geometric formulation is that, under transformations of the coordinates, i.e.\ the parameters of the family of probability distributions, a Riemannian metric transforms covariantly as a tensor.  Therefore, the information metric can be expressed as
\be
\df s^2 = \sum_{A = 1}^{K + 1} \sum_{B = 1}^{K + 1} I_{A B} \, \df \theta_A \, \df \theta_B ,
\ee
where $A, B = 1, \ldots , K + 1$, and
\be
I_{A B} = \bb E \bigg( \fr{\pd \log f}{\pd \theta_A} \fr{\pd \log f}{\pd \theta_B} \bigg) = - \bb E \bigg( \fr{\pd^2 \log f}{\pd \theta_A \, \pd \theta_B} \bigg) .
\label{IAB}
\ee
One may also check the equivalence explicitly using the chain rule for differentiation.  The matrix with components $I_{A B}$ is degenerate, i.e.\ $\det I_{A B} = 0$, whereas the matrix with components $\mathcal{I}_{a b}$ is non-degenerate, i.e.\ $\det \mathcal{I}_{a b} \neq 0$.

For the inverse Schl\"{o}milch distribution, it is convenient to take $(\theta_1 , \ldots, \theta_K , \theta_{K + 1}) = (\beta_1 , \ldots , \beta_K, \tau)$.  One choice of independent coordinates is $\beta_1 , \ldots , \beta_{K - 1}$, in terms of which the fill-up value $\beta_K = 1 - \sum_{i = 1}^{K - 1} \beta_i$ is determined, so we may then take $(\eta_1 , \ldots , \eta_{K - 1} , \eta_K) = (\beta_1 , \ldots , \beta_{K - 1}, \tau)$.  A different choice of $\eta_1 , \ldots , \eta_K$ will be used to show that the information metric is hyperbolic space.

We first use the factorization of the probability density function $f(\vec x) = h(\vec x)/J_0$ to organize more easily the computation of the Fisher information.

\begin{lemma}
\label{decomposelemma}
The Fisher information of the Concrete distribution can be expressed as
\be
I_{A B} = \fr{\pd^2 \log J_0}{\pd \theta_A \, \pd \theta_B} - J_0^{-1} \int_{S_K} \! \df {^{K - 1}} x \,  h(\vec x) \, \fr{\pd^2 \log h (\vec x)}{\pd \theta_A \, \pd \theta_B} .
\label{FIdecompose}
\ee
\end{lemma}

\begin{proof}
From the second expression in \eq{IAB}, the Fisher information is
\be
I_{A B} = - \int_{S_K} \! \df {^{K - 1}} x \, f(\vec x) \, \fr{\pd^2 \log f(\vec x)}{\pd \theta_A \, \pd \theta_B} .
\ee
Since $f(\vec x) = J_0^{-1} h(\vec x)$, it follows that
\begin{align}
I_{A B} & = J^{-1} \int_{S_K} \! \df {^{K - 1}} x \, h(\vec x) \, \fr{\pd^2 \log J}{\pd \theta_A \, \pd \theta_B} - J_0^{-1} \int_{S_K} \! \df {^{K - 1}} x \, h(\vec x) \, \fr{\pd^2 \log h(\vec x)}{\pd \theta_A \, \pd \theta_B} \nnr
& = J_0^{-1} \fr{\pd^2 \log J_0}{\pd \theta_A \, \pd \theta_B} \int_{S_K} \! \df {^{K - 1}} x \, h(\vec x) - J_0^{-1} \int_{S_K} \! \df {^{K - 1}} x \, h(\vec x) \, \fr{\pd^2 \log h(\vec x)}{\pd \theta_A \, \pd \theta_B} \nnr
& = \fr{\pd^2 \log J_0}{\pd \theta_A \, \pd \theta_B} - J_0^{-1} \int_{S_K} \! \df {^{K - 1}} x \, h(\vec x) \, \fr{\pd^2 \log h(\vec x)}{\pd \theta_A \, \pd \theta_B} ,
\end{align}
as required.
\end{proof}

We now compute the information metric of the Concrete distribution.  Although the metric is $K$-dimensional, we initially present it in terms of $K + 1$ coordinates, one of which is redundant.

\begin{theorem}
The information metric of the Concrete distribution is
\begin{align}
\df s^2 & = \fr{(K - 1) (K \pi^2/6 + 1) + \textstyle \tf{1}{2} \sum_{i = 1}^K \sum_{j = 1}^K (\log \beta_i - \log \beta_j)^2}{(K + 1) \tau^2} \, \df \tau^2 \nnr
& \qquad + 2 \sum_{i = 1}^K \fr{\sum_{j = 1}^K \log \beta_j - K \log \beta_i}{(K + 1) \tau \beta_i} \, \df \beta_i \, \df \tau + \sum_{i = 1}^K \sum_{j = 1}^K \fr{K \delta_{i j} - 1}{(K + 1) \beta_i \beta_j} \, \df \beta_i \, \df \beta_j .
\label{metric}
\end{align}
\end{theorem}

\begin{proof}
By differentiating $h(\vec x)$, given in \eq{hk}, we obtain the first derivatives
\begin{align}
\fr{\pd \log h (\vec x)}{\pd \beta_i} & = - \fr{K}{k(\vec x) x_i^\tau} , & \fr{\pd \log h (\vec x)}{\pd \tau} & = \fr{K \sum_{k = 1}^K (\beta_k/x_k^\tau) \log x_k}{k (\vec x)} - \sum_{j = 1}^K \log x_j .
\end{align}
We then obtain the second derivatives
\begin{align}
\label{dijlogh}
\fr{\pd^2 \log h (\vec x)}{\pd \beta_i \, \pd \beta_j} & = \fr{K}{k(\vec x)^2 x_i^\tau x_j^\tau} , \\
\label{ditlogh}
\fr{\pd^2 \log h (\vec x)}{\pd \beta_i \, \pd \tau} & = K \bigg( \fr{\log x_i}{k(\vec x) x_i^\tau} - \fr{\sum_{k = 1}^K (\beta_k/x_k^\tau) \log x_k}{k(\vec x)^2 x_i^\tau} \bigg) , \\
\label{dttlogh}
\fr{\pd^2 \log h (\vec x)}{\pd \tau^2} & = K \bigg( \fr{[\sum_{l = 1}^K (\beta_l/x_l^\tau) \log x_l]^2}{k(\vec x)^2} - \fr{\sum_{l = 1}^K (\beta_l/x_l^\tau) (\log x_l)^2}{k(\vec x)} \bigg) .
\end{align}
Each of these three types of terms will then be multiplied by $h(\vec x)$ and integrated over the simplex $S_K$.

Firstly, using \eq{Jspecial} and \eq{dijlogh}, we have
\be
\int_{S_K} \! \df {^{K - 1}} x \, h(\vec x) \fr{\pd^2 \log h (\vec x)}{\pd \beta_i \, \pd \beta_j} = \fr{J_0 (1 + \delta_{i j})}{(K + 1) \beta_i \, \beta_j} .
\label{fisherij}
\ee

Secondly, from \eq{dikJ}, setting $i = m$ and $k = n$, we have
\be
\beta_k \int_{S_K} \! \df {^{K - 1}} x \, h(\vec x) \fr{\log x_i - \log x_k}{k(\vec x)^2 x_i^\tau x_k^\tau} = - \fr{J_0 (1 + \delta_{i k})}{K (K + 1) \beta_i \tau} (\log \beta_k - \log \beta_i) .
\ee
Summing over $k$ gives
\be
\int_{S_K} \! \df {^{K - 1}} x \, h(\vec x) \bigg( \fr{\log x_i}{k(\vec x) x_i^\tau} - \fr{\sum_{k = 1}^K (\beta_k/x_k^\tau) \log x_k}{k(\vec x)^2 x_i^\tau} \bigg) = - \fr{J_0 (\sum_{k = 1}^K \log \beta_k - K \log \beta_i)}{K (K + 1) \beta_i \tau} ,
\ee
so \eq{ditlogh} gives
\be
\int_{S_K} \! \df {^{K - 1}} x \, h(\vec x) \fr{\pd^2 \log h(\vec x)}{\pd \beta_i \, \pd \tau} = - \fr{J_0 (\sum_{k = 1}^K \log \beta_k - K \log \beta_i)}{(K + 1) \beta_i \tau} .
\label{fisherit}
\ee

Thirdly, from \eq{dikilJ}, setting $k = m$ and $l = n$, we have
\begin{align}
& \beta_k \beta_l \int_{S_K} \! \df {^{K - 1}} x \, \fr{h(\vec x) (\log x_i - \log x_k) (\log x_i - \log x_l)}{k(\vec x)^2 x_k^\tau x_l^\tau} \nnr
& = \fr{J_0}{K (K + 1) \tau^2} [1 - \delta_{i l} + (- 1 + \delta_{i k} + \delta_{i l} - 2 \delta_{k l} + \delta_{i k l}) (- 2 \log \beta_i + \log \beta_k + \log \beta_l) \nnr
& \qquad + (1 + \delta_{k l}) (\log \beta_i - \log \beta_k) (\log \beta_i - \log \beta_l) + (1 - \delta_{i k} - \delta_{i l} + 3 \delta_{k l} - 2 \delta_{i k l}) \pi^2/6] .
\end{align}
Summing over $k$ and $l$ gives
\begin{align}
& \int_{S_K} \! \df {^{K - 1}} x \, h(\vec x) \bigg[ (\log x_i)^2 - \fr{2 \log x_i}{k(\vec x)} \sum_{k = 1}^K \fr{\beta_k \log x_k}{x_k^\tau} + \fr{1}{k(\vec x)^2} \bigg( \sum_{k = 1}^K \fr{\beta_k \log x_k}{x_k^\tau} \bigg) ^2 \bigg] \nnr
& = \fr{J_0}{K (K + 1) \tau^2} [K (K - 1) + 2 K (K + 1) \log \beta_i - 2 (K + 1) \textstyle \sum_{k = 1}^K \log \beta_k \nnr
& \qquad + (- K \log \beta_i + \textstyle \sum_{k = 1}^K \log \beta_k)^2 + \sum_{k = 1}^K (- \log \beta_i + \log \beta_k)^2 + (K + 2) (K - 1) \pi^2/6] .
\label{klsum}
\end{align}
From \eq{dikikJ}, setting $k = n$, we have
\begin{align}
& \beta_m \beta_k \int_{S_K} \! \df {^{K - 1}} x \, \fr{h(\vec x) (\log x_i - \log x_k)^2}{k(\vec x)^2 x_m^\tau x_k^\tau} \nnr
& = \fr{J_0}{K (K + 1) \tau^2} [- 2 \delta_{i m} + 2 \delta_{k m} + 2 (- 1 + \delta_{i k} + \delta_{i m} - 2 \delta_{k m} + \delta_{i k m}) (- \log \beta_i + \log \beta_k) \nnr
& \qquad + (1 + \delta_{m k}) (\log \beta_i - \log \beta_k)^2 + (1 - \delta_{i k} + \delta_{k m} - \delta_{i k m}) \pi^2/3] .
\end{align}
Summing over $m$ and $k$ gives
\begin{align}
& \int_{S_K} \! \df {^{K - 1}} x \, h(\vec x) \bigg( (\log x_i)^2 - \fr{2 \log x_i}{k(\vec x)} \sum_{k = 1}^K \fr{\beta_k \log x_k}{x_k^\tau} + \fr{1}{k(\vec x)} \sum_{k = 1}^K \fr{\beta_k (\log x_k)^2}{x_k^\tau} \bigg) \nnr
& = \fr{J_0}{K (K + 1) \tau^2} [2 K (K + 1) \log \beta_i - 2 (K + 1) \textstyle \sum_{k = 1}^K \log \beta_k \nnr
& \qquad + (K + 1) \textstyle \sum_{k = 1}^K (- \log \beta_i + \log \beta_k)^2 + (K + 1) (K - 1) \pi^2/3] .
\label{mksum}
\end{align}
Subtracting \eq{klsum} from \eq{mksum}, we obtain
\begin{align}
& \int_{S_K} \! \df {^{K - 1}} x \, h(\vec x) \bigg( \fr{\sum_{k = 1}^K (\beta_k/x_k^\tau) (\log x_k)^2}{k(\vec x)} - \fr{[\sum_{k = 1}^K (\beta_k/x_k^\tau) \log x_k]^2}{k(\vec x)^2} \bigg) \nnr
& = \fr{J_0}{K (K + 1) \tau^2} [ K (K - 1) (\pi^2/6 - 1) + \textstyle \tf{1}{2}   \sum_{i = 1}^K \sum_{k = 1}^K (\log \beta_i - \log \beta_k)^2 ] ,
\end{align}
and so \eq{dttlogh} gives
\be
\int_{S_K} \! \df {^{K - 1}} x \, h(\vec x) \fr{\pd^2 \log h (\vec x)}{\pd \tau^2} = - \fr{J_0 [ K (K - 1) (\pi^2/6 - 1) + \textstyle \tf{1}{2}   \sum_{i = 1}^K \sum_{k = 1}^K (\log \beta_i - \log \beta_k)^2 ]}{(K + 1) \tau^2} .
\label{fishertt}
\ee

Finally, we compute
\begin{align}
\sum_{A = 1}^{K + 1} \sum_{B = 1}^{K + 1} \fr{\pd^2 \log J_0}{\pd \theta_A \, \pd \theta_B} \, \df \theta_A \, \df \theta_B = \fr{K - 1}{\tau^2} \, \df \tau^2 + \sum_{i = 1}^K \fr{\delta_{i j}}{\beta_i \, \beta_j} \, \df \beta_i \, \df \beta_j .
\label{ddlogJ0}
\end{align}
Substituting \eq{fisherij}, \eq{fisherit}, \eq{fishertt} and \eq{ddlogJ0} into \eq{FIdecompose} gives the Fisher information \eq{metric}.
\end{proof}

For practical use, it may be more helpful to use a $K \times K$ Fisher information matrix $\mathcal{I}_{a b}$, which we now provide.

\begin{corollary}
Regarding $\beta_1 , \ldots , \beta_{K - 1}$ as independent and $\beta_K$ expressed in terms of these independent parameters through $\beta_K = 1 - \sum_{i = 1}^{K - 1} \beta_i$, the components of the Fisher information are
\begin{align}
\mathcal{I}_{\tau \tau} & = \fr{(K - 1) (K \pi^2/6 + 1) + \textstyle \tf{1}{2} \sum_{i = 1}^K \sum_{j = 1}^K (\log \beta_i - \log \beta_j)^2}{(K + 1) \tau^2} , \nnr
\mathcal{I}_{i \tau} & = \fr{1}{(K + 1) \tau} \bigg( \fr{\sum_{j = 1}^K \log \beta_j - K \log \beta_i}{\beta_i} - \fr{\sum_{j = 1}^K \log \beta_j - K \log \beta_K}{\beta_K} \bigg) , \nnr
\mathcal{I}_{i j} & = \fr{1}{K + 1} \bigg( \fr{K \delta_{i j} - 1}{\beta_i \beta_j} + \fr{1}{\beta_i \beta_K} + \fr{1}{\beta_j \beta_K} + \fr{K - 1}{\beta_K^2} \bigg) ,
\end{align}
where $i, j = 1, \ldots , K - 1$.
\end{corollary}
	
\begin{proof}
Set $\beta_K = 1 - \sum_{i = 1}^{K - 1} \beta_i$ in the information metric \eq{metric}, so $\df \beta_K = - \sum_{i = 1}^{K - 1} \df \beta_i$.
\end{proof}

Location-scale distributions, parameterized by a location in $\bb R^K$ and a single scale, generally have hyperbolic information metrics $\bb H^{K + 1}$ \cite{sabobe}.  One explicit example is the $K$-variate normal distributions with a covariance matrix proportional to the identity matrix \cite{bensadon}.  We similarly find that the Concrete distribution has a hyperbolic information metric.  Note that hyperbolic space can also arise from distributions that are not location-scale distributions, such as the negative multinomial distribution \cite{ollcua}.

\begin{theorem}
The information metric of the Concrete distribution is hyperbolic space $\bb H^K$ with sectional curvature $-1/\ell^2$, where $\ell = \sr{(K - 1) (K \pi^2/6 + 1)/(K + 1)}$.
\end{theorem}

\begin{proof}
The information metric \eq{metric} can be written as
\be
\df s^2 = \ell^2 \fr{\df \tau^2}{\tau^2} + \fr{1}{2 (K + 1)} \sum_{i = 1}^K \sum_{j = 1}^K \bigg( - \fr{\df \beta_i}{\beta_i} + \fr{\df \beta_j}{\beta_j} + (\log \beta_i - \log \beta_j) \fr{\df \tau}{\tau} \bigg) ^2 ,
\ee
where $\ell = \sr{(K - 1) (K \pi^2/6 + 1)/(K + 1)}$.  Make the coordinate change
\begin{align}
\xi_a & = \fr{\log(\beta_a/\beta_K)}{\ell \tau} , & \eta_K & = \fr{1}{\tau} ,
\label{xicoordinates}
\end{align}
for $a = 1 , \ldots , K - 1$, so the coordinates $\xi_a$ take values on $\bb R$ and $\eta_K > 0$.  Then
\be
\ell \, \df \xi_a = \fr{1}{\tau} \bigg( \fr{\df \beta_a}{\beta_a} - \fr{\df \beta_K}{\beta_K} - \fr{\log \beta_a - \log \beta_K}{\tau} \, \df \tau \bigg) ,
\ee
so the information metric takes the form
\be
\df s^2 = \ell^2 \bigg( \fr{\df \eta_K^2}{\eta_K^2} + \fr{\textstyle \sum_{a = 1}^{K - 1} \df \xi_a^2 + \tf{1}{2} \sum_{a = 1}^{K - 1} \sum_{b = 1}^{K - 1} (\df \xi_a - \df \xi_b)^2}{(K + 1) \eta_K^2} \bigg) .
\ee
By a linear coordinate transformation on the $\xi_a$ coordinates to cartesian coordinates $\eta_a$ for Euclidean space $\bb R^{K - 1}$, the information metric can be expressed as
\be
\df s^2 = \ell^2 \bigg( \fr{\df \eta_K^2}{\eta_K^2} + \fr{\textstyle \sum_{a = 1}^{K - 1} \df \eta_a^2}{\eta_K^2} \bigg) ,
\ee
which is hyperbolic space $\bb H^K$ expressed in Poincar\'{e} half-space coordinates.  An example of such a coordinate transformation is
\be
\xi_a = \fr{\sr{K + 1} \eta_a}{\sr{K}} + \fr{\sr{K + 1} \textstyle \sum_{b = 1}^{K - 1} \eta_b }{\sr{K} (\sr{K} + 1)} ,
\ee
or conversely
\be
\eta_a = \fr{\sr{K} \xi_a}{\sr{K + 1}} - \fr{\sum_{b = 1}^{K - 1} \xi_b}{\sr{K + 1} (\sr{K} + 1)} ,
\label{etacoordinates}
\ee
for $a = 1 , \ldots , K - 1$.
\end{proof}

The Poincar\'{e} half-space coordinates correspond to an orthogonal parameterization of the Concrete distribution.  Hyperbolic space $\bb H^K$ has further possible orthogonal coordinates, corresponding to other orthogonal parameterizations.

Geodesics on hyperbolic space $\bb H^K$ are well-known, and there is a unique geodesic that connects any two points (see e.g.\ \cite{caflkepa}).  The geodesics are lines of constant $\eta_1 , \ldots , \eta_{K - 1}$ and also semicircles with centres on $\eta_K = 0$.  The length of the connecting geodesic provides a geodesic distance, Fisher--Rao distance, between two Concrete distributions.

\begin{proposition}
The geodesic distance between concrete distributions $C(\gvec \beta, \tau)$ and $C(\gvec \beta ', \tau')$ is
\be
d(C(\gvec \beta, \tau), C(\gvec \beta ', \tau')) = 2 \ell \sinh^{-1} \Bigg ( \fr{1}{2} \sr{ \bigg( \sr{\fr{\tau'}{\tau}} - \sr{\fr{\tau}{\tau'}} \bigg) ^2 + \fr{\sum_{i =1}^K \sum_{j = 1}^K (\Delta_i - \Delta_j)^2}{2 (K + 1) \ell^2} } \Bigg ) ,
\ee
where
\be
\Delta_i = \sr{\fr{\tau'}{\tau}} \log \beta_i - \sr{\fr{\tau}{\tau'}} \log \beta_i'
\ee
and $\ell = \sr{(K - 1) (K \pi^2/6 + 1)/(K + 1)}$.
\end{proposition}

\begin{proof}
For the half-space model of unit hyperbolic space $\bb H^K$, given in Poincar\'{e} half-space coordinates as $\df s^2 = \sum_{i = 1}^K \df \eta_i^2/\eta_K^2$, the geodesic distance between $\gvec \eta$ and $\gvec \eta'$ is (see e.g.\ p.\ 80 of \cite{iversen} or p.\ 126 of \cite{ratcliffe})
\be
d_{\mathbb H^K} (\gvec \eta, \gvec \eta') = 2 \sinh^{-1} \bigg( \fr{| \gvec \eta - \gvec \eta' |}{2 \sr{\eta_K \eta_K'}} \bigg) .
\label{HKdistance}
\ee
From the coordinate transformation of $\eta_K$ in \eq{xicoordinates}, we have
\begin{align}
\eta_K & = 1/\tau , & \eta_K' & = 1/\tau' ,
\label{etaK}
\end{align}
and so
\be
(\eta_K - \eta_K')^2 = \bigg( \fr{1}{\tau} - \fr{1}{\tau'} \bigg) ^2 .
\ee
From the coordinate transformations \eq{xicoordinates} and \eq{etacoordinates}, we have
\begin{align}
\sum_{a = 1}^{K - 1} (\eta_a - \eta_a')^2 & = \fr{1}{K + 1} \bigg[ K \sum_{a = 1}^{K - 1}  (\xi_a - \xi_a')^2 - \bigg( \sum_{b = 1}^{K - 1} (\xi_b - \xi_b') \bigg)^2 \bigg] \nnr
& = \fr{1}{(K + 1) \ell^2 \sr{\tau \tau'}} \bigg[ K \sum_{a = 1}^{K - 1}  (\Delta_a - \Delta_K)^2 - \bigg( \sum_{b = 1}^{K - 1} \Delta_b - (K - 1) \Delta_K \bigg)^2 \bigg] \nnr
& = \fr{1}{2 (K + 1) \ell^2 \sr{\tau \tau'}} \sum_{i =1}^K \sum_{j = 1}^K (\Delta_i - \Delta_j)^2 .
\end{align}
Summing these two contributions, we have
\be
| \gvec \eta - \gvec \eta' |^2 = \bigg( \fr{1}{\tau} - \fr{1}{\tau'} \bigg) ^2 + \fr{\sum_{i =1}^K \sum_{j = 1}^K (\Delta_i - \Delta_j)^2}{2 (K + 1) \ell^2 \sr{\tau \tau'}} .
\label{etanorm}
\ee
We substitute \eq{etaK} and \eq{etanorm} into the distance formula for unit hyperbolic space \eq{HKdistance}, and then scale by a factor of $\ell$ to obtain
\be
d(C(\gvec \beta, \tau), C(\gvec \beta ', \tau')) = \ell d_{\mathbb H^K} (\gvec \eta, \gvec \eta') ,
\ee
which is the geodesic distance between the distributions.
\end{proof}

In the case that $\gvec \beta = \gvec \beta'$, the geodesic distance simplifies to
\be
d(C(\gvec \beta, \tau), C(\gvec \beta, \tau')) = 2 \ell \sinh^{-1} \Bigg ( \fr{| \tau - \tau' |}{2 \sr{\tau \tau'}} \sr{ 1 + \fr{\sum_{i =1}^K \sum_{j = 1}^K (\log \beta_i - \log \beta_j)^2}{2 (K + 1) \ell^2} } \Bigg ) .
\ee
In the special case of equal temperatures, $\tau = \tau'$, the geodesic distance simplifies to
\be
d(C(\gvec \beta, \tau), C(\gvec \beta ', \tau)) = 2 \ell \sinh^{-1} \Bigg ( \fr{1}{2 \ell} \sr{ \fr{\sum_{i = 1}^K \sum_{j = 1}^K [\log (\beta_i/\beta_i') - \log (\beta_j/\beta_j')]^2}{2 (K + 1)} } \Bigg ) ,
\label{equaltdistance}
\ee
which is independent of $\tau$.  However, it is not possible to use this equal temperature formula to define a distance in the $\tau \rightarrow 0$ or $\tau \rightarrow \infty$ limits.

The geodesic distance for Concrete distributions generically diverges if one of the temperatures becomes zero or infinite, i.e.\ in the limit that one of the distributions becomes a discrete categorical distribution, so does not provide a well-defined distance in this case.  The geodesic distance between two categorical distributions with probability vectors $\gvec \beta$ and $\gvec \beta'$ that are properly normalized, i.e.\ $\sum_{i = 1}^K \beta_i = \sum_{i = 1}^K \beta_i' = 1$, is (see e.g.\ \cite{kass} and references within)
\be
d(\gvec \beta, \gvec \beta') = 2 \cos^{-1} \bigg( \sum_{i = 1}^K \sr{\beta_i \beta'_i} \bigg) .
\label{catdistance}
\ee
This distance is derived from the metric on the positive orthant of a sphere $S^{K - 1}$ of radius $\ell = 2$, which in canonical simplex coordinates $\gvec \beta$ is $\df s^2 = \sum_{i = 1}^K \df \beta_i^2/\beta_i$, also known in mathematical biology as the Shahshahani metric \cite{shahshahani}.  There is no inconsistency between the finiteness of the geodesic distance for categorical distributions and the divergence in the zero or infinite temperature limits for Concrete distributions, or between the distances \eq{equaltdistance} and \eq{catdistance}, since such limits are not strictly Concrete distributions and do not correspond to any point on the hyperbolic space $\bb H^K$ information manifold.

\end{document}